\documentclass[a4paper,12pt]{article}

\usepackage[a4paper,margin=27mm]{geometry}

\usepackage{amscd,amsmath,amssymb,amsthm}
\usepackage{graphicx}

\newtheorem{theorem}{Theorem}[section]

\newtheorem{proposition}[theorem]{Proposition}
\newtheorem{corollary}[theorem]{Corollary}

\theoremstyle{definition}
\newtheorem{definition}[theorem]{Definition}
\newtheorem{example}[theorem]{Example}

\theoremstyle{remark}

\begin{document}

\title{Loop Fundamental Quandles:
A Geometric Generalization of the Knot Quandle}

\author{
Seiichi Kamada\\
Department of Mathematics, Graduate School of Science,\\
The University of Osaka\\
Toyonaka, Osaka 560-0043, Japan\\
\texttt{kamada@math.sci.osaka-u.ac.jp}
}

\date{}

\maketitle

\begin{abstract}
The knot quandle, introduced independently by Joyce and Matveev,
admits a geometric description in terms of nooses.
Motivated by this description, we introduce the
\textit{loop fundamental quandle}
$LQ(X,A,\ast)$ of a pair $(X,A)$ with a fixed base point
$\ast\in X$. 
Its elements are represented by pairs consisting of a loop in $A$
and a path joining the base point of the loop to the fixed base
point $\ast$.
More generally, for a loop condition $P$ on $A$, invariant under
free homotopy, we define the loop fundamental quandle
$LQ(X,A,\ast;P)$.

We establish functoriality and homotopy invariance of the loop
fundamental quandle and show that it admits a natural augmented
quandle structure over $\pi_1(X,\ast)$.
For a knot $K$ in $S^3$, applying the construction to its exterior
gives the \textit{knot loop quandle} $LQ(K,\ast;P)$.
Appropriate choices of the loop condition recover the knot quandle,
its $n$-th power quandles, and the full knot quandle.
The construction is inspired by Joyce's noose-theoretic approach,
but allows loop data that need not arise as the boundary of a disk
satisfying the noose condition.
\end{abstract}

\medskip


\smallskip



\section{Introduction}\label{sec:introduction}

The knot quandle, introduced independently by Joyce
\cite{Joyce82} and Matveev \cite{Matveev82}, is a powerful
invariant of knots.
Joyce gave both a combinatorial description of the knot quandle
based on knot diagrams and a geometric description in terms of
nooses.
He also introduced the fundamental quandle
$Q(X,A,\ast)$ of a pair consisting of a topological space $X$ and
a subspace $A$, with a fixed base point $\ast$.

The purpose of this paper is to introduce another quandle
associated with a pair of topological spaces.
Our construction is motivated by the geometric picture underlying
Joyce's noose construction.
Instead of taking a disk together with its tail as the basic
geometric object, we consider a pair
\[
(\ell,\alpha),
\]
where $\ell$ is a loop in $A$ and $\alpha$ is a path in $X$ from
the base point of $\ell$ to the fixed base point $\ast$.
After passing to an appropriate homotopy relation, the quandle
operation is given by
\[
[\ell_1,\alpha_1]\rhd[\ell_2,\alpha_2]
=
[\ell_1,
\alpha_1\alpha_2^{-1}\ell_2\alpha_2].
\]
We call the resulting quandle the
\textit{loop fundamental quandle} of $(X,A)$ based at $\ast$ and
denote it by
\[
LQ(X,A,\ast).
\]

A useful feature of this construction is that one can impose a
condition on the loop component.
More precisely, if $P$ is a set of loops in $A$ invariant under
free homotopy, which we call a \textit{loop condition}, we obtain
a quandle
\[
LQ(X,A,\ast;P).
\]
The freedom to choose $P$ plays an important role in the
applications to knot theory considered in this paper.

We also study some basic structural properties of the loop
fundamental quandle.
It is functorial with respect to maps of based pairs, and homotopic
maps induce the same quandle homomorphism.
In particular, a homotopy equivalence of based pairs induces an
isomorphism of loop fundamental quandles.
We further show that the loop fundamental quandle has a natural
augmented quandle structure associated with the fundamental group
$\pi_1(X,\ast)$.

For a knot $K$ in $S^3$, let
\[
E(K)=S^3\setminus\operatorname{Int}N(K)
\]
be its exterior and let
\[
A_K=\partial E(K).
\]
For a loop condition $P$ on $A_K$, we define the
\textit{knot loop quandle}
\[
LQ(K,\ast;P)
=
LQ(E(K),A_K,\ast;P).
\]
When $K$ is oriented and $P_1$ is the loop condition consisting
of loops freely homotopic to a positive meridian, the knot loop
quandle $LQ(K,\ast;P_1)$ is naturally isomorphic to the usual knot
quandle $Q(K,\ast)$.
More generally, if $P_n$ consists of loops freely homotopic to the
$n$-th power of a positive meridian, then
\[
LQ(K,\ast;P_n)
\cong
Q(K,\ast)^{[n]},
\]
where $Q(K,\ast)^{[n]}$ denotes the $n$-th power quandle of the
knot quandle.
More general loop conditions give further knot loop quandles; in
particular, the loop condition corresponding to
\[
\{(1,0),(-1,0)\}\subset\mathbb Z^2
\]
gives the full knot quandle.

One distinction between the present construction and Joyce's
noose construction is that a loop used in a loop fundamental
quandle need not arise as the boundary of a disk satisfying the
noose condition.
Thus the loop fundamental quandle provides a flexible extension
of the geometric viewpoint underlying the knot quandle.
This also suggests analogous constructions for codimension-two
submanifolds in more general manifolds.

The paper is organized as follows.
In Section~\ref{sec:loop-fundamental-quandle}, we define the loop
fundamental quandle, introduce loop conditions, and establish the
basic quandle properties.
In Section~\ref{sec:functoriality}, we study functoriality,
homotopy invariance, and invariance under homotopy equivalence.
In Section~\ref{sec:augmented-quandle}, we show that the loop
fundamental quandle carries a natural augmented quandle structure.
In Section~\ref{sec:joyce-knot-quandle}, we recall Joyce's
fundamental quandle and the knot quandle in terms of nooses.
In Section~\ref{sec:knot-loop-quandle}, we introduce the knot loop
quandle and relate it to the knot quandle, its power quandles, and
the full knot quandle.
Finally, in Section~\ref{sec:concluding-remarks}, we compare the
present construction with Joyce's noose construction and discuss
possible extensions to codimension-two submanifolds.


\section{The loop fundamental quandle}
\label{sec:loop-fundamental-quandle}

We begin by recalling the definition of a quandle.

\begin{definition}
A \textit{quandle} is a set $Q$ equipped with two binary operations
$\rhd$ and $\overline{\rhd}$ satisfying the following conditions:
\begin{enumerate}
\item[(i)]
For any $a \in Q$,
\[
a \rhd a = a.
\]

\item[(ii)]
For any $a,b \in Q$,
\[
(a \rhd b)\,\overline{\rhd}\,b = a
\quad\text{and}\quad
(a\,\overline{\rhd}\,b)\rhd b = a.
\]

\item[(iii)]
For any $a,b,c \in Q$,
\[
(a \rhd b)\rhd c
=
(a \rhd c)\rhd(b \rhd c).
\]
\end{enumerate}
\end{definition}

The operation $\overline{\rhd}$ is uniquely determined by $\rhd$.
Thus, a quandle may be regarded either as a triple
$(Q,\rhd,\overline{\rhd})$ or simply as a pair $(Q,\rhd)$.
We often denote the quandle $(Q,\rhd,\overline{\rhd})$ simply by $Q$.

Let $(X,A)$ be a pair consisting of a topological space $X$ and
a subspace $A$ of $X$.
Fix a base point $\ast \in X$; we do not require that $\ast$ belong
to $A$.
Let $I=[0,1]$.

Throughout this paper, if $\alpha$ and $\beta$ are paths satisfying
$\alpha(1)=\beta(0)$, then $\alpha\beta$ denotes their concatenation,
traversed first along $\alpha$ and then along $\beta$.
For a path $\alpha$, we denote by $\alpha^{-1}$ the path with the
opposite orientation, defined by $\alpha^{-1}(t)=\alpha(1-t)$.

Let $L\Pi(X,A,\ast)$ denote the set of all pairs $(\ell,\alpha)$,
where
\[
\ell:I\to A
\quad\text{and}\quad
\alpha:I\to X
\]
are paths satisfying
\[
\ell(0)=\ell(1)=\alpha(0)
\quad\text{and}\quad
\alpha(1)=\ast.
\]
Thus, $\ell$ is a loop in $A$, and $\alpha$ is a path in $X$
from the base point of $\ell$ to the fixed base point $\ast$ of $X$.

We introduce an equivalence relation on $L\Pi(X,A,\ast)$ as follows.
Two elements $(\ell,\alpha)$ and $(\ell',\alpha')$ are said to be
\textit{homotopic} if there exist homotopies
\[
H:I\times I\to A
\quad\text{and}\quad
K:I\times I\to X
\]
such that
\[
H(t,0)=\ell(t),\qquad
H(t,1)=\ell'(t),
\]
\[
K(t,0)=\alpha(t),\qquad
K(t,1)=\alpha'(t),
\]
and, for every $s\in I$,
\[
H(0,s)=H(1,s)=K(0,s)
\quad\text{and}\quad
K(1,s)=\ast.
\]
In other words, the pair
$(H(\,\cdot\,,s),K(\,\cdot\,,s))$ belongs to
$L\Pi(X,A,\ast)$ for every $s\in I$.

We denote the equivalence class of $(\ell,\alpha)$ by
\[
[(\ell,\alpha)]
\quad\text{or simply}\quad
[\ell,\alpha],
\]
and let
\[
L\pi(X,A,\ast)
=
L\Pi(X,A,\ast)/{\sim}
\]
be the set of equivalence classes.

We now define two binary operations on $L\Pi(X,A,\ast)$.
Let
\[
(\ell_1,\alpha_1),(\ell_2,\alpha_2)\in L\Pi(X,A,\ast).
\]
We set
\[
(\ell_1,\alpha_1)\rhd(\ell_2,\alpha_2)
=
\bigl(
\ell_1,
\alpha_1\alpha_2^{-1}\ell_2\alpha_2
\bigr)
\]
and
\[
(\ell_1,\alpha_1)\,\overline{\rhd}\,(\ell_2,\alpha_2)
=
\bigl(
\ell_1,
\alpha_1\alpha_2^{-1}\ell_2^{-1}\alpha_2
\bigr).
\]

We shall show that these operations are compatible with the above
equivalence relation and hence induce well-defined binary operations
on $L\pi(X,A,\ast)$, given by
\[
[\ell_1,\alpha_1]\rhd[\ell_2,\alpha_2]
=
[\ell_1,\alpha_1\alpha_2^{-1}\ell_2\alpha_2]
\]
and
\[
[\ell_1,\alpha_1]\,\overline{\rhd}\,[\ell_2,\alpha_2]
=
[\ell_1,\alpha_1\alpha_2^{-1}\ell_2^{-1}\alpha_2].
\]

\begin{theorem}\label{theorem:mainA}
The two operations defined above are well defined on
$L\pi(X,A,\ast)$, and
$L\pi(X,A,\ast)$, equipped with the operations
$\rhd$ and $\overline{\rhd}$, is a quandle.
\end{theorem}

\begin{definition}
Let $(X,A)$ be a pair consisting of a topological space $X$ and
a subspace $A$ of $X$, and let $\ast$ be a base point of $X$.
We call the quandle $(L\pi(X,A,\ast),\rhd,\overline{\rhd})$ the
\textit{loop fundamental quandle of $(X,A)$ based at $\ast$},
and denote it by
\[
LQ(X,A,\ast).
\]
\end{definition}

Let $L(A)$ denote the set of all loops in $A$.
Two loops $\ell$ and $\ell'$ in $A$ are said to be
\textit{freely homotopic} if they are homotopic through loops in $A$,
where the base point of the loop is allowed to move during the homotopy.

\begin{definition}
A subset $P$ of $L(A)$ is called a \textit{loop condition on $A$}
if it is invariant under free homotopy; that is, for any freely
homotopic loops $\ell,\ell'\in L(A)$,
\[
\ell\in P
\quad\Longrightarrow\quad
\ell'\in P.
\]
\end{definition}

Let $P$ be a loop condition on $A$.
Define a subset $L\Pi(X,A,\ast;P)$ of $L\Pi(X,A,\ast)$ by
\[
L\Pi(X,A,\ast;P)
=
\{
(\ell,\alpha)\in L\Pi(X,A,\ast)
\mid
\ell\in P
\}.
\]
If $(\ell,\alpha)$ and $(\ell',\alpha')$ are homotopic in
$L\Pi(X,A,\ast)$ and $\ell\in P$, then $\ell$ and $\ell'$ are
freely homotopic in $A$.
Hence $\ell'\in P$.
It follows that the subset $L\pi(X,A,\ast;P)$ of
$L\pi(X,A,\ast)$ given by
\[
L\pi(X,A,\ast;P)
=
\{
[\ell,\alpha]\in L\pi(X,A,\ast)
\mid
\ell\in P
\}
\]
is well defined.

\begin{theorem}\label{theorem:mainB}
The operations $\rhd$ and $\overline{\rhd}$ on
$L\pi(X,A,\ast)$ restrict to binary operations on
$L\pi(X,A,\ast;P)$.
Moreover,
$L\pi(X,A,\ast;P)$, equipped with these operations, is a quandle.
\end{theorem}

\begin{definition}
Let $(X,A)$ be a pair consisting of a topological space $X$ and
a subspace $A$ of $X$, let $\ast$ be a base point of $X$, and let
$P$ be a loop condition on $A$.
We call the quandle $(L\pi(X,A,\ast;P),\rhd,\overline{\rhd})$ the
\textit{loop fundamental quandle of $(X,A)$ based at $\ast$
with loop condition $P$}, and denote it by
\[
LQ(X,A,\ast;P).
\]
Thus,
\[
LQ(X,A,\ast;P)
=
\bigl(
L\pi(X,A,\ast;P),
\rhd,
\overline{\rhd}
\bigr).
\]
\end{definition}

We now prove Theorems~\ref{theorem:mainA} and~\ref{theorem:mainB}.

\begin{proof}[Proof of Theorem~\ref{theorem:mainA}]
First, we show that $\rhd$ and $\overline{\rhd}$ are well-defined
operations on $L\pi(X,A,\ast)$.

Suppose that
\[
(\ell_i,\alpha_i)\sim(\ell_i',\alpha_i')
\qquad (i=1,2).
\]
We first show that
\[
(\ell_1,\alpha_1)\rhd(\ell_2,\alpha_2)
\sim
(\ell_1',\alpha_1')\rhd(\ell_2',\alpha_2').
\]

For $i=1,2$, let
\[
H_i:I\times I\to A
\quad\text{and}\quad
K_i:I\times I\to X
\]
be homotopies such that
\[
H_i(t,0)=\ell_i(t),\qquad
H_i(t,1)=\ell_i'(t),
\]
\[
K_i(t,0)=\alpha_i(t),\qquad
K_i(t,1)=\alpha_i'(t),
\]
and, for every $s\in I$,
\[
H_i(0,s)=H_i(1,s)=K_i(0,s)
\quad\text{and}\quad
K_i(1,s)=\ast.
\]
In other words, the pair
\[
\bigl(H_i(\,\cdot\,,s),K_i(\,\cdot\,,s)\bigr)
\]
belongs to $L\Pi(X,A,\ast)$ for every $s\in I$.

For each $s\in I$, put
\[
\ell_{i,s}=H_i(\,\cdot\,,s)
\quad\text{and}\quad
\alpha_{i,s}=K_i(\,\cdot\,,s)
\qquad (i=1,2).
\]
Then
\[
(\ell_{i,s},\alpha_{i,s})\in L\Pi(X,A,\ast)
\]
for every $s\in I$.

Consider the family
\[
\bigl(
\ell_{1,s},
\alpha_{1,s}\alpha_{2,s}^{-1}\ell_{2,s}\alpha_{2,s}
\bigr),
\qquad s\in I.
\]
For every $s\in I$, this pair belongs to $L\Pi(X,A,\ast)$.
Indeed, the initial point of the path
\[
\alpha_{1,s}\alpha_{2,s}^{-1}\ell_{2,s}\alpha_{2,s}
\]
is
\[
\alpha_{1,s}(0)=\ell_{1,s}(0),
\]
and its terminal point is $\ast$.

When $s=0$, this pair is
\[
\bigl(
\ell_1,
\alpha_1\alpha_2^{-1}\ell_2\alpha_2
\bigr)
=
(\ell_1,\alpha_1)\rhd(\ell_2,\alpha_2),
\]
whereas when $s=1$, it is
\[
\bigl(
\ell_1',
(\alpha_1')(\alpha_2')^{-1}\ell_2'\alpha_2'
\bigr)
=
(\ell_1',\alpha_1')\rhd(\ell_2',\alpha_2').
\]
Thus the above family gives a homotopy in $L\Pi(X,A,\ast)$ from
$(\ell_1,\alpha_1)\rhd(\ell_2,\alpha_2)$ to
$(\ell_1',\alpha_1')\rhd(\ell_2',\alpha_2')$.
Therefore
\[
(\ell_1,\alpha_1)\rhd(\ell_2,\alpha_2)
\sim
(\ell_1',\alpha_1')\rhd(\ell_2',\alpha_2').
\]

Similarly, the family
\[
\bigl(
\ell_{1,s},
\alpha_{1,s}\alpha_{2,s}^{-1}
\ell_{2,s}^{-1}\alpha_{2,s}
\bigr),
\qquad s\in I,
\]
gives a homotopy in $L\Pi(X,A,\ast)$ from
\[
(\ell_1,\alpha_1)\,\overline{\rhd}\,(\ell_2,\alpha_2)
\]
to
\[
(\ell_1',\alpha_1')\,\overline{\rhd}\,
(\ell_2',\alpha_2').
\]
Hence
\[
(\ell_1,\alpha_1)\,\overline{\rhd}\,(\ell_2,\alpha_2)
\sim
(\ell_1',\alpha_1')\,\overline{\rhd}\,
(\ell_2',\alpha_2').
\]

Thus both operations descend to well-defined binary operations on
$L\pi(X,A,\ast)$.

We now verify the quandle axioms.

\medskip
\noindent
{\rm (i)}
Let $[\ell,\alpha]\in L\pi(X,A,\ast)$. We have
\[
(\ell,\alpha)\rhd(\ell,\alpha)
=
(\ell,\alpha\alpha^{-1}\ell\alpha)
\sim
(\ell,\ell\alpha).
\]
Here and below, cancellation of a path followed by its inverse is
understood up to homotopy relative to the endpoints.

Moreover,
\[
(\ell,\ell\alpha)\sim(\ell,\alpha).
\]
Indeed, this homotopy is obtained by moving the base point of the
loop $\ell$ once around $\ell$, while shortening the initial copy
of $\ell$ in the path $\ell\alpha$ accordingly.
Therefore
\[
[\ell,\alpha]\rhd[\ell,\alpha]
=
[\ell,\alpha].
\]

\medskip
\noindent
{\rm (ii)}
For
$[\ell_1,\alpha_1],[\ell_2,\alpha_2]\in L\pi(X,A,\ast)$,
we have
\[
\begin{aligned}
&
\bigl(
(\ell_1,\alpha_1)\rhd(\ell_2,\alpha_2)
\bigr)
\,\overline{\rhd}\,
(\ell_2,\alpha_2)
\\
&\qquad =
\bigl(
\ell_1,
(\alpha_1\alpha_2^{-1}\ell_2\alpha_2)
\alpha_2^{-1}\ell_2^{-1}\alpha_2
\bigr)
\\
&\qquad \sim
(\ell_1,\alpha_1).
\end{aligned}
\]
Hence
\[
\bigl(
[\ell_1,\alpha_1]\rhd[\ell_2,\alpha_2]
\bigr)
\,\overline{\rhd}\,
[\ell_2,\alpha_2]
=
[\ell_1,\alpha_1].
\]
Similarly,
\[
\bigl(
[\ell_1,\alpha_1]\,\overline{\rhd}\,[\ell_2,\alpha_2]
\bigr)
\rhd
[\ell_2,\alpha_2]
=
[\ell_1,\alpha_1].
\]

\medskip
\noindent
{\rm (iii)}
Let
\[
x_i=[\ell_i,\alpha_i]
\qquad (i=1,2,3).
\]
We have
\[
\begin{aligned}
(x_1\rhd x_2)\rhd x_3
&=
\bigl[
\ell_1,
(\alpha_1\alpha_2^{-1}\ell_2\alpha_2)
\alpha_3^{-1}\ell_3\alpha_3
\bigr]
\\
&=
\bigl[
\ell_1,
\alpha_1\alpha_2^{-1}\ell_2\alpha_2
\alpha_3^{-1}\ell_3\alpha_3
\bigr].
\end{aligned}
\]

On the other hand,
\[
x_1\rhd x_3
=
[\ell_1,\alpha_1\alpha_3^{-1}\ell_3\alpha_3]
\]
and
\[
x_2\rhd x_3
=
[\ell_2,\alpha_2\alpha_3^{-1}\ell_3\alpha_3].
\]
Therefore
\[
\begin{aligned}
&(x_1\rhd x_3)\rhd(x_2\rhd x_3)
\\
&=
\Bigl[
\ell_1,
(\alpha_1\alpha_3^{-1}\ell_3\alpha_3)
(\alpha_2\alpha_3^{-1}\ell_3\alpha_3)^{-1}
\ell_2
(\alpha_2\alpha_3^{-1}\ell_3\alpha_3)
\Bigr]
\\
&=
\Bigl[
\ell_1,
\alpha_1\alpha_3^{-1}\ell_3\alpha_3
\alpha_3^{-1}\ell_3^{-1}\alpha_3\alpha_2^{-1}
\ell_2
\alpha_2\alpha_3^{-1}\ell_3\alpha_3
\Bigr]
\\
&=
\bigl[
\ell_1,
\alpha_1\alpha_2^{-1}\ell_2\alpha_2
\alpha_3^{-1}\ell_3\alpha_3
\bigr].
\end{aligned}
\]
Thus
\[
(x_1\rhd x_2)\rhd x_3
=
(x_1\rhd x_3)\rhd(x_2\rhd x_3).
\]

Hence $L\pi(X,A,\ast)$ satisfies all the quandle axioms.
\end{proof}

\begin{proof}[Proof of Theorem~\ref{theorem:mainB}]
By Theorem~\ref{theorem:mainA},
$L\pi(X,A,\ast)$ is a quandle with the operations
$\rhd$ and $\overline{\rhd}$.

Let
\[
[\ell_1,\alpha_1],
[\ell_2,\alpha_2]
\in
L\pi(X,A,\ast;P).
\]
Then
\[
\ell_1,\ell_2\in P.
\]
By definition,
\[
[\ell_1,\alpha_1]\rhd[\ell_2,\alpha_2]
=
[\ell_1,
\alpha_1\alpha_2^{-1}\ell_2\alpha_2].
\]
The first component of the representative on the right-hand side
is $\ell_1$, which belongs to $P$.
Hence
\[
[\ell_1,\alpha_1]\rhd[\ell_2,\alpha_2]
\in
L\pi(X,A,\ast;P).
\]

Similarly,
\[
[\ell_1,\alpha_1]\,\overline{\rhd}\,[\ell_2,\alpha_2]
=
[\ell_1,
\alpha_1\alpha_2^{-1}\ell_2^{-1}\alpha_2],
\]
and again the first component is $\ell_1\in P$.
Therefore
\[
[\ell_1,\alpha_1]\,\overline{\rhd}\,[\ell_2,\alpha_2]
\in
L\pi(X,A,\ast;P).
\]

Thus $L\pi(X,A,\ast;P)$ is closed under both operations
$\rhd$ and $\overline{\rhd}$, and hence is a subquandle of
$L\pi(X,A,\ast)$.
\end{proof}

We next consider a natural map obtained by forgetting the loop
component.

Let $(X,A)$ be a pair consisting of a topological space $X$ and a
subspace $A$ of $X$, and let $\ast$ be a base point of $X$.

We denote by $\Pi(X,A,\ast)$ the set of all paths
\[
\alpha:I\longrightarrow X
\]
such that
\[
\alpha(0)\in A
\quad\text{and}\quad
\alpha(1)=\ast.
\]
We say that two paths
$\alpha,\alpha'\in\Pi(X,A,\ast)$ are \textit{homotopic}
if there exists a homotopy
\[
K:I\times I\longrightarrow X
\]
such that
\[
K(t,0)=\alpha(t),
\qquad
K(t,1)=\alpha'(t),
\]
and, for every $s\in I$,
\[
K(0,s)\in A
\quad\text{and}\quad
K(1,s)=\ast.
\]
In other words, the path $K(\,\cdot\,,s)$ belongs to
$\Pi(X,A,\ast)$ for every $s\in I$.

We denote by $[\alpha]$ the equivalence class of $\alpha$ under
this homotopy and let
\[
\pi(X,A,\ast)
=
\Pi(X,A,\ast)/{\sim}
\]
be the set of equivalence classes.

Consider the projection onto the second component
\[
\varpi_\Pi:
L\Pi(X,A,\ast)
\longrightarrow
\Pi(X,A,\ast),
\qquad
\varpi_\Pi(\ell,\alpha)=\alpha.
\]
If
\[
(\ell,\alpha)\sim(\ell',\alpha')
\]
in $L\Pi(X,A,\ast)$, then
\[
\alpha\sim\alpha'
\]
in $\Pi(X,A,\ast)$.
Hence $\varpi_\Pi$ induces a well-defined map
\[
\varpi:
L\pi(X,A,\ast)
\longrightarrow
\pi(X,A,\ast),
\qquad
\varpi([\ell,\alpha])=[\alpha].
\]
We call $\varpi$ the \textit{loop-forgetting map}.

Let $P_0$ be the loop condition on $A$ consisting of all loops
that are freely null-homotopic in $A$.
For each $x\in A$, let
\[
c_x:I\longrightarrow A,
\qquad
c_x(t)=x,
\]
be the constant loop at $x$.

Every element of $L\pi(X,A,\ast;P_0)$ has a representative of the
form
\[
[c_x,\alpha],
\]
where
\[
\alpha\in\Pi(X,A,\ast)
\quad\text{and}\quad
\alpha(0)=x.
\]
Thus
\[
L\pi(X,A,\ast;P_0)
=
\{
[c_x,\alpha]
\mid
\alpha\in\Pi(X,A,\ast),\ \alpha(0)=x
\}.
\]
Indeed, a free null-homotopy of the loop component, together
with the corresponding motion of the initial point of the path
component, gives such a representative.

\begin{proposition}\label{prop:loop-forgetting-bijection}
The restriction of the loop-forgetting map,
\[
\varpi_0:
L\pi(X,A,\ast;P_0)
\longrightarrow
\pi(X,A,\ast),
\]
is a bijection.
\end{proposition}

\begin{proof}
Define a map
\[
\sigma:
\pi(X,A,\ast)
\longrightarrow
L\pi(X,A,\ast;P_0)
\]
by
\[
\sigma([\alpha])
=
[c_{\alpha(0)},\alpha].
\]

We first show that $\sigma$ is well defined.
Suppose that
\[
[\alpha]=[\alpha']
\]
in $\pi(X,A,\ast)$.
Let
\[
K:I\times I\longrightarrow X
\]
be a homotopy from $\alpha$ to $\alpha'$ such that, for every
$s\in I$,
\[
K(0,s)\in A
\quad\text{and}\quad
K(1,s)=\ast.
\]
Define
\[
H:I\times I\longrightarrow A
\]
by
\[
H(t,s)=K(0,s).
\]
Then, for every $s\in I$, $H(\,\cdot\,,s)$ is the constant loop
at $K(0,s)$, and
\[
H(0,s)=H(1,s)=K(0,s).
\]
Thus the pair $(H,K)$ gives a homotopy in
$L\Pi(X,A,\ast;P_0)$ from
\[
(c_{\alpha(0)},\alpha)
\]
to
\[
(c_{\alpha'(0)},\alpha').
\]
Hence
\[
[c_{\alpha(0)},\alpha]
=
[c_{\alpha'(0)},\alpha'],
\]
and therefore $\sigma$ is well defined.

For any $[\alpha]\in\pi(X,A,\ast)$, we have
\[
(\varpi_0\circ\sigma)([\alpha])
=
\varpi_0([c_{\alpha(0)},\alpha])
=
[\alpha].
\]
Thus
\[
\varpi_0\circ\sigma
=
\operatorname{id}_{\pi(X,A,\ast)}.
\]

On the other hand, let
\[
[c_x,\alpha]
\in
L\pi(X,A,\ast;P_0),
\]
where $\alpha(0)=x$.
Then
\[
\begin{aligned}
(\sigma\circ\varpi_0)([c_x,\alpha])
&=
\sigma([\alpha])
\\
&=
[c_{\alpha(0)},\alpha]
\\
&=
[c_x,\alpha].
\end{aligned}
\]
Hence
\[
\sigma\circ\varpi_0
=
\operatorname{id}_{L\pi(X,A,\ast;P_0)}.
\]

Therefore $\varpi_0$ is a bijection with inverse $\sigma$.
\end{proof}

In particular, the loop-forgetting map
\[
\varpi:
L\pi(X,A,\ast)
\longrightarrow
\pi(X,A,\ast)
\]
is surjective.


\section{Functoriality and homotopy invariance}
\label{sec:functoriality}

Let $(X,A)$ and $(Y,B)$ be pairs consisting of topological spaces
and their subspaces, and let $\ast_X\in X$ and $\ast_Y\in Y$ be
base points.

By a map
\[
f:(X,A,\ast_X)\longrightarrow(Y,B,\ast_Y)
\]
we mean a continuous map $f:X\to Y$ such that
\[
f(A)\subset B
\quad\text{and}\quad
f(\ast_X)=\ast_Y.
\]

\begin{proposition}\label{prop:induced-homomorphism}
Let
\[
f:(X,A,\ast_X)\longrightarrow(Y,B,\ast_Y)
\]
be a map as above.
Then $f$ induces a well-defined map
\[
f_\ast:
L\pi(X,A,\ast_X)
\longrightarrow
L\pi(Y,B,\ast_Y)
\]
defined by
\[
f_\ast([\ell,\alpha])
=
[f\circ\ell,f\circ\alpha].
\]
Moreover, $f_\ast$ is a quandle homomorphism
\[
f_\ast:
LQ(X,A,\ast_X)
\longrightarrow
LQ(Y,B,\ast_Y).
\]
\end{proposition}

\begin{proof}
Let
\[
(\ell,\alpha)\in L\Pi(X,A,\ast_X).
\]
Since $\ell(I)\subset A$ and $f(A)\subset B$, we have
\[
(f\circ\ell)(I)\subset B.
\]
Moreover,
\[
(f\circ\ell)(0)
=
(f\circ\ell)(1)
=
(f\circ\alpha)(0)
\]
and
\[
(f\circ\alpha)(1)
=
f(\ast_X)
=
\ast_Y.
\]
Hence
\[
(f\circ\ell,f\circ\alpha)
\in
L\Pi(Y,B,\ast_Y).
\]

Suppose that
\[
(\ell,\alpha)\sim(\ell',\alpha')
\]
in $L\Pi(X,A,\ast_X)$.
Let
\[
H:I\times I\to A
\quad\text{and}\quad
K:I\times I\to X
\]
be homotopies realizing this equivalence.
Then
\[
f\circ H:I\times I\to B
\quad\text{and}\quad
f\circ K:I\times I\to Y
\]
give a homotopy from
$(f\circ\ell,f\circ\alpha)$ to
$(f\circ\ell',f\circ\alpha')$ in
$L\Pi(Y,B,\ast_Y)$.
Thus
\[
(f\circ\ell,f\circ\alpha)
\sim
(f\circ\ell',f\circ\alpha'),
\]
and hence $f_\ast$ is well defined.

It remains to verify that $f_\ast$ preserves the quandle
operations.
For
\[
x=[\ell_1,\alpha_1],
\qquad
y=[\ell_2,\alpha_2],
\]
we have
\[
\begin{aligned}
f_\ast(x\rhd y)
&=
f_\ast
\bigl(
[\ell_1,
\alpha_1\alpha_2^{-1}\ell_2\alpha_2]
\bigr)
\\
&=
[
f\circ\ell_1,
(f\circ\alpha_1)
(f\circ\alpha_2)^{-1}
(f\circ\ell_2)
(f\circ\alpha_2)
]
\\
&=
f_\ast(x)\rhd f_\ast(y).
\end{aligned}
\]
Similarly,
\[
f_\ast
\bigl(
x\,\overline{\rhd}\,y
\bigr)
=
f_\ast(x)\,\overline{\rhd}\,f_\ast(y).
\]
Therefore $f_\ast$ is a quandle homomorphism.
\end{proof}

We now consider loop conditions.
Let $(X,A)$ and $(Y,B)$ be pairs consisting of topological spaces
and their subspaces, and let $\ast_X\in X$ and $\ast_Y\in Y$ be
base points.
Let $P_A$ and $P_B$ be loop conditions on $A$ and $B$,
respectively.

For a map $f:X\to Y$, put
\[
f_\sharp(P_A)
=
\{f\circ\ell\mid \ell\in P_A\}.
\]
By a map
\[
f:(X,A,\ast_X;P_A)
\longrightarrow
(Y,B,\ast_Y;P_B)
\]
we mean a continuous map $f:X\to Y$ such that
\[
f(A)\subset B,
\qquad
f(\ast_X)=\ast_Y,
\qquad
f_\sharp(P_A)\subset P_B.
\]

\begin{proposition}\label{prop:induced-homomorphism-loop-condition}
Let
\[
f:(X,A,\ast_X;P_A)
\longrightarrow
(Y,B,\ast_Y;P_B)
\]
be a map as above.
Then $f$ induces a well-defined map
\[
f_\ast:
L\pi(X,A,\ast_X;P_A)
\longrightarrow
L\pi(Y,B,\ast_Y;P_B)
\]
defined by
\[
f_\ast([\ell,\alpha])
=
[f\circ\ell,f\circ\alpha].
\]
Moreover, $f_\ast$ is a quandle homomorphism
\[
f_\ast:
LQ(X,A,\ast_X;P_A)
\longrightarrow
LQ(Y,B,\ast_Y;P_B).
\]
\end{proposition}

\begin{proof}
Let
\[
[\ell,\alpha]
\in
L\pi(X,A,\ast_X;P_A).
\]
Then $\ell\in P_A$.
Since
\[
f_\sharp(P_A)\subset P_B,
\]
we have
\[
f\circ\ell\in P_B.
\]
Moreover, by Proposition~\ref{prop:induced-homomorphism},
\[
[f\circ\ell,f\circ\alpha]
\in
L\pi(Y,B,\ast_Y).
\]
Therefore
\[
[f\circ\ell,f\circ\alpha]
\in
L\pi(Y,B,\ast_Y;P_B).
\]
Thus the map
\[
f_\ast:
L\pi(X,A,\ast_X;P_A)
\longrightarrow
L\pi(Y,B,\ast_Y;P_B)
\]
is well defined.

Since the map on
$L\pi(X,A,\ast_X;P_A)$ is the restriction of the quandle
homomorphism
\[
f_\ast:
LQ(X,A,\ast_X)
\longrightarrow
LQ(Y,B,\ast_Y)
\]
given in Proposition~\ref{prop:induced-homomorphism},
it preserves both quandle operations.
Hence
\[
f_\ast:
LQ(X,A,\ast_X;P_A)
\longrightarrow
LQ(Y,B,\ast_Y;P_B)
\]
is a quandle homomorphism.
\end{proof}

\begin{proposition}\label{prop:functoriality-loop-condition}
Let
\[
f:(X,A,\ast_X;P_A)
\longrightarrow
(Y,B,\ast_Y;P_B)
\]
and
\[
g:(Y,B,\ast_Y;P_B)
\longrightarrow
(Z,C,\ast_Z;P_C)
\]
be maps as above.
Let
\[
f_\ast:
LQ(X,A,\ast_X;P_A)
\longrightarrow
LQ(Y,B,\ast_Y;P_B)
\]
and
\[
g_\ast:
LQ(Y,B,\ast_Y;P_B)
\longrightarrow
LQ(Z,C,\ast_Z;P_C)
\]
be the induced quandle homomorphisms.
Then the following hold:
\begin{enumerate}
\item[(i)]
If $1_X:X\to X$ denotes the identity map, then
\[
(1_X)_\ast
=
\operatorname{id}_{LQ(X,A,\ast_X;P_A)}.
\]

\item[(ii)]
We have
\[
(g\circ f)_\ast
=
g_\ast\circ f_\ast.
\]
\end{enumerate}
\end{proposition}

\begin{proof}
The identity map $1_X:X\to X$ satisfies
\[
(1_X)_\sharp(P_A)=P_A,
\]
and hence is a map
\[
1_X:
(X,A,\ast_X;P_A)
\longrightarrow
(X,A,\ast_X;P_A).
\]
For any
\[
[\ell,\alpha]\in LQ(X,A,\ast_X;P_A),
\]
we have
\[
(1_X)_\ast([\ell,\alpha])
=
[1_X\circ\ell,1_X\circ\alpha]
=
[\ell,\alpha].
\]
Therefore
\[
(1_X)_\ast
=
\operatorname{id}_{LQ(X,A,\ast_X;P_A)}.
\]

For the composition, since
\[
f_\sharp(P_A)\subset P_B
\quad\text{and}\quad
g_\sharp(P_B)\subset P_C,
\]
we have
\[
(g\circ f)_\sharp(P_A)
=
g_\sharp\bigl(f_\sharp(P_A)\bigr)
\subset P_C.
\]
Thus
\[
g\circ f:
(X,A,\ast_X;P_A)
\longrightarrow
(Z,C,\ast_Z;P_C)
\]
is a map of based pairs with loop conditions.

For any
\[
[\ell,\alpha]\in LQ(X,A,\ast_X;P_A),
\]
we have
\[
\begin{aligned}
(g\circ f)_\ast([\ell,\alpha])
&=
[(g\circ f)\circ\ell,(g\circ f)\circ\alpha]
\\
&=
[g\circ(f\circ\ell),g\circ(f\circ\alpha)]
\\
&=
g_\ast([f\circ\ell,f\circ\alpha])
\\
&=
(g_\ast\circ f_\ast)([\ell,\alpha]).
\end{aligned}
\]
Hence
\[
(g\circ f)_\ast
=
g_\ast\circ f_\ast.
\]
\end{proof}

\begin{corollary}\label{cor:functoriality}
Let
\[
f:(X,A,\ast_X)\longrightarrow(Y,B,\ast_Y)
\]
and
\[
g:(Y,B,\ast_Y)\longrightarrow(Z,C,\ast_Z)
\]
be maps of based pairs.
Let
\[
f_\ast:
LQ(X,A,\ast_X)\longrightarrow LQ(Y,B,\ast_Y)
\]
and
\[
g_\ast:
LQ(Y,B,\ast_Y)\longrightarrow LQ(Z,C,\ast_Z)
\]
be the induced quandle homomorphisms.
Then the following hold:
\begin{enumerate}
\item[(i)]
If $1_X:X\to X$ denotes the identity map, then
\[
(1_X)_\ast
=
\operatorname{id}_{LQ(X,A,\ast_X)}.
\]

\item[(ii)]
We have
\[
(g\circ f)_\ast
=
g_\ast\circ f_\ast.
\]
\end{enumerate}
\end{corollary}

\begin{proof}
Apply Proposition~\ref{prop:functoriality-loop-condition} with
\[
P_A=L(A),
\qquad
P_B=L(B),
\qquad
P_C=L(C).
\]
Since
\[
LQ(X,A,\ast_X;L(A))=LQ(X,A,\ast_X),
\]
and similarly for $(Y,B,\ast_Y)$ and $(Z,C,\ast_Z)$,
the result follows immediately.
\end{proof}

Thus the assignment
\[
(X,A,\ast;P)
\longmapsto
LQ(X,A,\ast;P)
\]
is functorial with respect to maps of based pairs with loop
conditions. In particular, the assignment
\[
(X,A,\ast)\longmapsto LQ(X,A,\ast)
\]
is functorial with respect to maps of based pairs.

We next discuss homotopies of maps.

We say that two maps
\[
f,g:(X,A,\ast_X)\longrightarrow(Y,B,\ast_Y)
\]
are \textit{homotopic} if there exists a homotopy
\[
F:X\times I\longrightarrow Y
\]
such that
\[
F(x,0)=f(x),
\qquad
F(x,1)=g(x)
\]
for every $x\in X$, and, for every $s\in I$,
\[
F(A,s)\subset B
\quad\text{and}\quad
F(\ast_X,s)=\ast_Y.
\]
Equivalently, if we put
\[
f_s=F(\,\cdot\,,s),
\]
then
\[
f_s:(X,A,\ast_X)\longrightarrow(Y,B,\ast_Y)
\]
is a map for every $s\in I$.

Let $P_A$ and $P_B$ be loop conditions on $A$ and $B$,
respectively.
We say that two maps
\[
f,g:(X,A,\ast_X;P_A)
\longrightarrow
(Y,B,\ast_Y;P_B)
\]
are \textit{homotopic} if there exists a homotopy
\[
F:X\times I\longrightarrow Y
\]
such that, for every $s\in I$, the map
\[
f_s=F(\,\cdot\,,s)
\]
is a map
\[
f_s:(X,A,\ast_X;P_A)
\longrightarrow
(Y,B,\ast_Y;P_B).
\]

\begin{proposition}\label{prop:homotopy-invariance}
Suppose that
\[
f,g:(X,A,\ast_X;P_A)
\longrightarrow
(Y,B,\ast_Y;P_B)
\]
are homotopic.
Then
\[
f_\ast=g_\ast:
LQ(X,A,\ast_X;P_A)
\longrightarrow
LQ(Y,B,\ast_Y;P_B).
\]
\end{proposition}

\begin{proof}
Let
\[
F:X\times I\longrightarrow Y
\]
be a homotopy from $f$ to $g$, and put
\[
f_s=F(\,\cdot\,,s)
\qquad (s\in I).
\]
By assumption, each $f_s$ is a map
\[
f_s:(X,A,\ast_X;P_A)
\longrightarrow
(Y,B,\ast_Y;P_B).
\]

Let
\[
[\ell,\alpha]\in
LQ(X,A,\ast_X;P_A).
\]
For each $s\in I$, we have
\[
(f_s\circ\ell,f_s\circ\alpha)
\in
L\Pi(Y,B,\ast_Y;P_B).
\]
Moreover, as $s$ varies from $0$ to $1$, the family
\[
(f_s\circ\ell,f_s\circ\alpha)
\]
gives a homotopy in
$L\Pi(Y,B,\ast_Y;P_B)$ from
\[
(f\circ\ell,f\circ\alpha)
\]
to
\[
(g\circ\ell,g\circ\alpha).
\]
Hence
\[
[f\circ\ell,f\circ\alpha]
=
[g\circ\ell,g\circ\alpha].
\]
Therefore
\[
f_\ast([\ell,\alpha])
=
g_\ast([\ell,\alpha]).
\]
Since this holds for every
$[\ell,\alpha]\in LQ(X,A,\ast_X;P_A)$, we obtain
\[
f_\ast=g_\ast.
\]
\end{proof}

\begin{corollary}\label{cor:homotopy-invariance}
Suppose that
\[
f,g:(X,A,\ast_X)
\longrightarrow
(Y,B,\ast_Y)
\]
are homotopic.
Then
\[
f_\ast=g_\ast:
LQ(X,A,\ast_X)
\longrightarrow
LQ(Y,B,\ast_Y).
\]
\end{corollary}

\begin{proof}
Apply Proposition~\ref{prop:homotopy-invariance} with
\[
P_A=L(A)
\quad\text{and}\quad
P_B=L(B),
\]
using
\[
LQ(X,A,\ast_X;L(A))=LQ(X,A,\ast_X)
\]
and
\[
LQ(Y,B,\ast_Y;L(B))=LQ(Y,B,\ast_Y).
\]
\end{proof}

We say that a map
\[
f:(X,A,\ast_X)\longrightarrow(Y,B,\ast_Y)
\]
is a \textit{homotopy equivalence} if there exists a map
\[
g:(Y,B,\ast_Y)\longrightarrow(X,A,\ast_X)
\]
such that
\[
g\circ f
\simeq
1_X
\]
as maps
\[
(X,A,\ast_X)\longrightarrow(X,A,\ast_X),
\]
and
\[
f\circ g
\simeq
1_Y
\]
as maps
\[
(Y,B,\ast_Y)\longrightarrow(Y,B,\ast_Y).
\]
In this case, we say that
$(X,A,\ast_X)$ and $(Y,B,\ast_Y)$ are
\textit{homotopy equivalent} and write
\[
(X,A,\ast_X)\simeq(Y,B,\ast_Y).
\]

Similarly, we say that a map
\[
f:(X,A,\ast_X;P_A)
\longrightarrow
(Y,B,\ast_Y;P_B)
\]
is a \textit{homotopy equivalence} if there exists a map
\[
g:(Y,B,\ast_Y;P_B)
\longrightarrow
(X,A,\ast_X;P_A)
\]
such that
\[
g\circ f
\simeq
1_X
\]
as maps of based pairs with loop conditions, and
\[
f\circ g
\simeq
1_Y
\]
as maps of based pairs with loop conditions.
In this case, we say that
$(X,A,\ast_X;P_A)$ and $(Y,B,\ast_Y;P_B)$ are
\textit{homotopy equivalent} and write
\[
(X,A,\ast_X;P_A)
\simeq
(Y,B,\ast_Y;P_B).
\]

\begin{proposition}\label{prop:homotopy-equivalence-invariance}
Suppose that
\[
f:(X,A,\ast_X;P_A)
\longrightarrow
(Y,B,\ast_Y;P_B)
\]
is a homotopy equivalence.
Then the induced quandle homomorphism
\[
f_\ast:
LQ(X,A,\ast_X;P_A)
\longrightarrow
LQ(Y,B,\ast_Y;P_B)
\]
is an isomorphism.
\end{proposition}

\begin{proof}
Let
\[
g:(Y,B,\ast_Y;P_B)
\longrightarrow
(X,A,\ast_X;P_A)
\]
be a homotopy inverse of $f$.
Then
\[
g\circ f\simeq 1_X
\quad\text{and}\quad
f\circ g\simeq 1_Y.
\]
By Proposition~\ref{prop:homotopy-invariance},
\[
(g\circ f)_\ast=(1_X)_\ast
\quad\text{and}\quad
(f\circ g)_\ast=(1_Y)_\ast.
\]
By Proposition~\ref{prop:functoriality-loop-condition},
\[
(g\circ f)_\ast=g_\ast\circ f_\ast
\quad\text{and}\quad
(f\circ g)_\ast=f_\ast\circ g_\ast.
\]
Moreover,
\[
(1_X)_\ast
=
\operatorname{id}_{LQ(X,A,\ast_X;P_A)}
\]
and
\[
(1_Y)_\ast
=
\operatorname{id}_{LQ(Y,B,\ast_Y;P_B)}.
\]
Hence
\[
g_\ast\circ f_\ast
=
\operatorname{id}_{LQ(X,A,\ast_X;P_A)}
\]
and
\[
f_\ast\circ g_\ast
=
\operatorname{id}_{LQ(Y,B,\ast_Y;P_B)}.
\]
Therefore $f_\ast$ is a quandle isomorphism with inverse $g_\ast$.
\end{proof}

\begin{corollary}\label{cor:homotopy-equivalence-invariance}
Suppose that
\[
f:(X,A,\ast_X)
\longrightarrow
(Y,B,\ast_Y)
\]
is a homotopy equivalence.
Then the induced quandle homomorphism
\[
f_\ast:
LQ(X,A,\ast_X)
\longrightarrow
LQ(Y,B,\ast_Y)
\]
is an isomorphism.
\end{corollary}

\begin{proof}
Apply Proposition~\ref{prop:homotopy-equivalence-invariance} with
\[
P_A=L(A)
\quad\text{and}\quad
P_B=L(B).
\]
\end{proof}

\section{The loop fundamental quandle as an augmented quandle}
\label{sec:augmented-quandle}

We first recall the notion of an augmented quandle.

\begin{proposition}[\cite{Joyce82}, \cite{FennRourke92}]
\label{prop:def-augmented-quandle}
Let $Q$ be a set and let $G$ be a group acting on $Q$ from the right,
\[
Q\times G\longrightarrow Q,
\qquad
(q,g)\longmapsto qg.
\]
Let
\[
\varepsilon:Q\longrightarrow G
\]
be a map satisfying the following conditions:
\begin{enumerate}
\item[(i)]
For every $q\in Q$,
\[
q\varepsilon(q)=q.
\]

\item[(ii)]
For every $q\in Q$ and $g\in G$,
\[
\varepsilon(qg)
=
g^{-1}\varepsilon(q)g.
\]
\end{enumerate}
Define binary operations $\rhd$ and $\overline{\rhd}$ on $Q$ by
\[
a\rhd b
=
a\varepsilon(b)
\]
and
\[
a\,\overline{\rhd}\,b
=
a\varepsilon(b)^{-1}.
\]
Then
\[
(Q,\rhd,\overline{\rhd})
\]
is a quandle.
\end{proposition}

The quandle $(Q,\rhd,\overline{\rhd})$, together with the right
action of $G$ on $Q$ and the map
\[
\varepsilon:Q\longrightarrow G,
\]
is called an \textit{augmented quandle}.
See \cite{Joyce82,FennRourke92}.

Let $(X,A)$ be a pair consisting of a topological space $X$ and a
subspace $A$ of $X$, let $\ast$ be a base point of $X$, and let
$P$ be a loop condition on $A$.

\begin{proposition}
\label{prop:interpretation-as-augmented-quandle}
Let
\[
Q=L\pi(X,A,\ast;P)
\quad\text{and}\quad
G=\pi_1(X,\ast).
\]
Define a right action of $G$ on $Q$ by
\[
[\ell,\alpha]\,[\gamma]
=
[\ell,\alpha\gamma],
\]
where
\[
[\ell,\alpha]\in Q
\quad\text{and}\quad
[\gamma]\in G.
\]
Define also a map
\[
\varepsilon:Q\longrightarrow G
\]
by
\[
\varepsilon([\ell,\alpha])
=
[\alpha^{-1}\ell\alpha].
\]

Then the following statements hold:
\begin{enumerate}
\item[(i)]
The above right action of $G$ on $Q$ and the map
$\varepsilon:Q\to G$ are well defined, and they satisfy
conditions {\rm (i)} and {\rm (ii)} of
Proposition~\ref{prop:def-augmented-quandle}.

\item[(ii)]
The quandle structure on $Q$ induced by this augmented quandle
structure coincides with the loop fundamental quandle structure
on
\[
LQ(X,A,\ast;P).
\]
\end{enumerate}
\end{proposition}

\begin{proof}
We first verify that the right action is well defined.
Suppose that
\[
[\ell,\alpha]
=
[\ell',\alpha']
\]
in $Q$ and
\[
[\gamma]=[\gamma']
\]
in $G$.
Let
\[
H:I\times I\to A
\quad\text{and}\quad
K:I\times I\to X
\]
be homotopies realizing
\[
(\ell,\alpha)\sim(\ell',\alpha').
\]
Appending a homotopy from $\gamma$ to $\gamma'$ to the paths
$K(\,\cdot\,,s)$ gives a homotopy from
\[
(\ell,\alpha\gamma)
\]
to
\[
(\ell',\alpha'\gamma').
\]
Hence
\[
[\ell,\alpha\gamma]
=
[\ell',\alpha'\gamma'],
\]
so the action is well defined.

The identities
\[
q\,1=q
\]
and
\[
(qg)h=q(gh)
\]
follow immediately from the corresponding properties of
concatenation of paths, up to homotopy relative to endpoints.
Thus this defines a right action of $G$ on $Q$.

We next show that $\varepsilon$ is well defined.
Suppose that
\[
[\ell,\alpha]
=
[\ell',\alpha']
\]
in $Q$.
Let
\[
H:I\times I\to A
\quad\text{and}\quad
K:I\times I\to X
\]
be homotopies realizing
\[
(\ell,\alpha)\sim(\ell',\alpha').
\]
For each $s\in I$, the loop
\[
K(\,\cdot\,,s)^{-1}
H(\,\cdot\,,s)
K(\,\cdot\,,s)
\]
is based at $\ast$.
As $s$ varies from $0$ to $1$, these loops give a based homotopy
from
\[
\alpha^{-1}\ell\alpha
\]
to
\[
(\alpha')^{-1}\ell'\alpha'.
\]
Therefore
\[
[\alpha^{-1}\ell\alpha]
=
[(\alpha')^{-1}\ell'\alpha']
\]
in $\pi_1(X,\ast)$.
Hence $\varepsilon$ is well defined.

Now let
\[
q=[\ell,\alpha]\in Q.
\]
Then
\[
\begin{aligned}
q\varepsilon(q)
&=
[\ell,\alpha]\,[\alpha^{-1}\ell\alpha]
\\
&=
[\ell,\alpha\alpha^{-1}\ell\alpha]
\\
&=
[\ell,\ell\alpha]
\\
&=
[\ell,\alpha]
\\
&=
q.
\end{aligned}
\]
Thus condition {\rm (i)} of
Proposition~\ref{prop:def-augmented-quandle} is satisfied.

Next, let
\[
q=[\ell,\alpha]\in Q
\quad\text{and}\quad
g=[\gamma]\in G.
\]
Then
\[
\begin{aligned}
\varepsilon(qg)
&=
\varepsilon([\ell,\alpha\gamma])
\\
&=
[(\alpha\gamma)^{-1}\ell(\alpha\gamma)]
\\
&=
[\gamma^{-1}\alpha^{-1}\ell\alpha\gamma]
\\
&=
[\gamma]^{-1}
[\alpha^{-1}\ell\alpha]
[\gamma]
\\
&=
g^{-1}\varepsilon(q)g.
\end{aligned}
\]
Hence condition {\rm (ii)} of
Proposition~\ref{prop:def-augmented-quandle} is also satisfied.
This proves (i).

We now prove (ii).
To distinguish the two quandle structures on the set
\[
Q=L\pi(X,A,\ast;P),
\]
let $\rhd'$ and $\overline{\rhd}'$ denote the operations induced
by the augmented quandle structure, while
$\rhd$ and $\overline{\rhd}$ denote the operations of the loop
fundamental quandle.

Let
\[
[\ell_1,\alpha_1],
[\ell_2,\alpha_2]
\in Q.
\]
Then
\[
\begin{aligned}
\relax[\ell_1,\alpha_1]
\rhd'
[\ell_2,\alpha_2]
&=
[\ell_1,\alpha_1]\,
\varepsilon([\ell_2,\alpha_2])
\\
&=
[\ell_1,\alpha_1]\,
[\alpha_2^{-1}\ell_2\alpha_2]
\\
&=
[\ell_1,
\alpha_1\alpha_2^{-1}\ell_2\alpha_2]
\\
&=
[\ell_1,\alpha_1]
\rhd
[\ell_2,\alpha_2].
\end{aligned}
\]
Similarly,
\[
\begin{aligned}
\relax[\ell_1,\alpha_1]
\,\overline{\rhd}'\,
[\ell_2,\alpha_2]
&=
[\ell_1,\alpha_1]\,
\varepsilon([\ell_2,\alpha_2])^{-1}
\\
&=
[\ell_1,\alpha_1]\,
[\alpha_2^{-1}\ell_2\alpha_2]^{-1}
\\
&=
[\ell_1,
\alpha_1\alpha_2^{-1}\ell_2^{-1}\alpha_2]
\\
&=
[\ell_1,\alpha_1]
\,\overline{\rhd}\,
[\ell_2,\alpha_2].
\end{aligned}
\]
Therefore
\[
\rhd'=\rhd
\quad\text{and}\quad
\overline{\rhd}'=\overline{\rhd}.
\]
Hence the quandle structure induced by the augmented quandle
structure is exactly the loop fundamental quandle
$LQ(X,A,\ast;P)$.
\end{proof}

\section{Joyce's fundamental quandle and the knot quandle}
\label{sec:joyce-knot-quandle}

In this section, we first recall Joyce's fundamental quandle of a pair
of topological spaces and then recall the knot quandle of a knot.

Let
\[
D^2=\{(x,y)\in\mathbb{R}^2\mid x^2+y^2\leq 1\}
\]
be the closed unit disk, endowed with the orientation induced by the
standard orientation of $\mathbb{R}^2$.
Let $O=(0,0)$.
We regard the boundary $\partial D^2$ as a loop based at $(1,0)$,
oriented counterclockwise.
More precisely, let
\[
\delta:I\to\partial D^2,\qquad
\delta(t)=(\cos 2\pi t,\sin 2\pi t).
\]

Let
\[
T=\{(x,0)\in\mathbb{R}^2\mid 1\leq x\leq 5\},
\]
and let
\[
N=D^2\cup T.
\]

Let $(X,A)$ be a pair consisting of a topological space $X$ and
a subspace $A$ of $X$, and choose a base point
$\ast\in X\setminus A$.

By a \textit{noose for $(X,A)$ based at $\ast$}, we mean a map
\[
\nu:N\to X
\]
such that
\[
\nu^{-1}(A)=\{O\}
\quad\text{and}\quad
\nu(5,0)=\ast.
\]

Let $\nu:N\to X$ be a noose for $(X,A)$ based at $\ast$.
Associated with $\nu$ are maps
\[
d_\nu:D^2\to X
\quad\text{and}\quad
\alpha_\nu:I\to X,
\]
where $d_\nu$ is the restriction of $\nu$ to $D^2$, and
$\alpha_\nu$ is defined by
\[
\alpha_\nu(t)=\nu(4t+1,0).
\]
Thus,
\[
d_\nu(1,0)=\alpha_\nu(0),
\qquad
\alpha_\nu(1)=\ast.
\]

Conversely, suppose that maps
\[
d:D^2\to X
\quad\text{and}\quad
\alpha:I\to X
\]
satisfy
\[
d^{-1}(A)=\{O\},\qquad
\alpha^{-1}(A)=\emptyset,
\]
\[
d(1,0)=\alpha(0),
\qquad
\alpha(1)=\ast.
\]
Then there exists a unique noose $\nu$ for $(X,A)$ based at $\ast$
such that
\[
d=d_\nu
\quad\text{and}\quad
\alpha=\alpha_\nu.
\]
We call $\nu$ the \textit{noose determined by $d$ and $\alpha$}.

Two nooses $\nu$ and $\nu'$ for $(X,A)$ based at $\ast$ are said
to be \textit{equivalent}, or \textit{homotopic}, if there exists
a homotopy
\[
F:N\times I\to X
\]
from $\nu$ to $\nu'$ such that, for every $s\in I$, the map
\[
F_s:N\to X,\qquad F_s(z)=F(z,s),
\]
is a noose for $(X,A)$ based at $\ast$.

Let $\mathcal{Q}(X,A,\ast)$ denote the set of equivalence classes
of nooses for $(X,A)$ based at $\ast$.

For a map $d:D^2\to X$, define the loop
\[
\partial d:I\to X
\]
by
\[
\partial d=d\circ\delta.
\]

For two nooses $\nu_i$ $(i=1,2)$, let
\[
d_i=d_{\nu_i}:D^2\to X,
\qquad
\alpha_i=\alpha_{\nu_i}:I\to X.
\]
We define $\nu_1\rhd\nu_2$ to be the noose determined by
$d_1$ and
\[
\alpha_1\alpha_2^{-1}(\partial d_2)\alpha_2.
\]
Similarly, we define
$\nu_1\,\overline{\rhd}\,\nu_2$ to be the noose determined by
$d_1$ and
\[
\alpha_1\alpha_2^{-1}(\partial d_2)^{-1}\alpha_2.
\]

Joyce \cite{Joyce82} proved that, for any two elements
$[\nu_1]$ and $[\nu_2]$ of $\mathcal{Q}(X,A,\ast)$, the formulas
\[
[\nu_1]\rhd[\nu_2]
=
[\nu_1\rhd\nu_2]
\]
and
\[
[\nu_1]\,\overline{\rhd}\,[\nu_2]
=
[\nu_1\,\overline{\rhd}\,\nu_2]
\]
define well-defined binary operations on
$\mathcal{Q}(X,A,\ast)$, and that
$\mathcal{Q}(X,A,\ast)$ equipped with these operations is a quandle.

\begin{definition}[\cite{Joyce82}]
Let $(X,A)$ be a pair consisting of a topological space $X$ and
a subspace $A$ of $X$, and let $\ast\in X\setminus A$.
The \textit{fundamental quandle of the pair $(X,A)$ based at $\ast$}
is the quandle
\[
Q(X,A,\ast)
=
\bigl(
\mathcal{Q}(X,A,\ast),
\rhd,
\overline{\rhd}
\bigr).
\]
\end{definition}

Let $K$ be an oriented knot in the $3$-sphere $S^3$, and let
$\ast\in S^3\setminus K$.

A \textit{positive} (respectively, \textit{negative})
\textit{meridional noose for $K$} is a noose for $(S^3,K)$
based at $\ast$ such that
\[
d_\nu:D^2\to S^3
\]
is a positive (respectively, negative) meridian disk of the
oriented knot $K$.

Let $Q(K,\ast)$ be the subset of $Q(S^3,K,\ast)$ consisting of
equivalence classes that admit representatives which are positive
meridional nooses for $K$.
Then $Q(K,\ast)$ is closed under the quandle operations of
$Q(S^3,K,\ast)$; in other words, $Q(K,\ast)$ is a subquandle of
$Q(S^3,K,\ast)$.

\begin{definition}[\cite{Joyce82}]
Let $K$ be an oriented knot in $S^3$, and let
$\ast\in S^3\setminus K$.
The \textit{knot quandle of $K$ based at $\ast$} is the quandle
$Q(K,\ast)$.
\end{definition}

Let $\widetilde{Q}(K,\ast)$ be the subset of $Q(S^3,K,\ast)$
consisting of equivalence classes that admit representatives which
are either positive or negative meridional nooses for $K$.
This subset is independent of the choice of orientation of $K$,
and hence $\widetilde{Q}(K,\ast)$ is well defined even when $K$
is unoriented.
Moreover, it is closed under the quandle operations of
$Q(S^3,K,\ast)$.

\begin{definition}
Let $K$ be a knot in $S^3$, and let
$\ast\in S^3\setminus K$.
The \textit{full knot quandle of $K$ based at $\ast$} is the quandle
$\widetilde{Q}(K,\ast)$.
\end{definition}

The full knot quandle $\widetilde{Q}(K,\ast)$ admits a natural
involution
\[
\rho:\widetilde{Q}(K,\ast)\longrightarrow
\widetilde{Q}(K,\ast),
\]
which interchanges the classes represented by positive and negative
meridional nooses.
This involution is a \textit{good involution}.
The full knot quandle equipped with this good involution,
\[
\bigl(\widetilde{Q}(K,\ast),\rho\bigr),
\]
is called the \textit{knot symmetric quandle}.
See \cite{Kamada07,KamadaOshiro10,Kamada14} for symmetric quandles
and their applications to knot theory.

\section{The knot loop quandle: A generalization of the knot quandle}
\label{sec:knot-loop-quandle}

Let $K$ be a knot in $S^3$, and let $N(K)$ be a closed tubular
neighborhood of $K$.
Let
\[
E(K)=S^3\setminus\operatorname{Int}N(K)
\]
be the exterior of $K$, and fix a base point $\ast\in E(K)$.
Let
\[
A_K=\partial E(K)=\partial N(K).
\]

Let $P$ be a loop condition on $A_K$.

\begin{definition}
The \textit{knot loop quandle of $K$ based at $\ast$ with loop
condition $P$}, denoted by $LQ(K,\ast;P)$, is the loop fundamental
quandle
\[
LQ(K,\ast;P)
=
LQ(E(K),A_K,\ast;P)
=
\bigl(
L\pi(E(K),A_K,\ast;P),
\rhd,
\overline{\rhd}
\bigr).
\]
\end{definition}

Let
\[
\varpi:
L\pi(E(K),A_K,\ast)
\longrightarrow
\pi(E(K),A_K,\ast),
\qquad
\varpi([\ell,\alpha])=[\alpha],
\]
be the loop-forgetting map defined in
Section~\ref{sec:loop-fundamental-quandle}.
For a loop condition $P$ on $A_K$, denote by
\[
\varpi_P:
LQ(K,\ast;P)
\longrightarrow
\pi(E(K),A_K,\ast)
\]
the restriction of $\varpi$ to $LQ(K,\ast;P)$.

Let $P_0$ be the loop condition on $A_K$ consisting of all loops
that are freely null-homotopic in $A_K$.

\begin{proposition}\label{prop:knot-loop-forgetting-bijection}
The loop-forgetting map
\[
\varpi_0:
LQ(K,\ast;P_0)
\longrightarrow
\pi(E(K),A_K,\ast)
\]
is a bijection.
\end{proposition}

\begin{proof}
This follows immediately from
Proposition~\ref{prop:loop-forgetting-bijection}
by taking
\[
X=E(K)
\quad\text{and}\quad
A=A_K.
\]
\end{proof}

Suppose now that $K$ is oriented.
Let $P_1$ be the set of loops in $A_K$ that are freely homotopic
to a positive meridian loop of $K$.
Then $P_1$ is a loop condition on $A_K$.

\begin{proposition}\label{prop:knot-loop-forgetting-bijection-P1}
The loop-forgetting map
\[
\varpi_1:
LQ(K,\ast;P_1)
\longrightarrow
\pi(E(K),A_K,\ast)
\]
is a bijection.
\end{proposition}

\begin{proof}
For each
\[
\alpha\in\Pi(E(K),A_K,\ast),
\]
let $\ell_\alpha$ be a positive meridian loop based at
$\alpha(0)$.
Since
\[
A_K\cong S^1\times S^1
\]
has abelian fundamental group, the free homotopy class of a
positive meridian determines a unique based homotopy class of
positive meridian loops at each point of $A_K$.
Thus the class
\[
[\ell_\alpha,\alpha]
\in
L\pi(E(K),A_K,\ast;P_1)
\]
does not depend on the choice of $\ell_\alpha$ within its based
homotopy class.

Define
\[
\sigma_1:
\pi(E(K),A_K,\ast)
\longrightarrow
L\pi(E(K),A_K,\ast;P_1)
\]
by
\[
\sigma_1([\alpha])
=
[\ell_\alpha,\alpha].
\]

We show that $\sigma_1$ is well defined.
Suppose that $\alpha$ and $\alpha'$ are homotopic in
\[
\Pi(E(K),A_K,\ast).
\]
Let
\[
K:I\times I\longrightarrow E(K)
\]
be a homotopy from $\alpha$ to $\alpha'$ such that
\[
K(0,s)\in A_K
\quad\text{and}\quad
K(1,s)=\ast
\]
for every $s\in I$.
The path
\[
\beta(s)=K(0,s)
\]
moves the base point from $\alpha(0)$ to $\alpha'(0)$.
Transporting the positive meridian $\ell_\alpha$ along $\beta$
gives a positive meridian loop based at $\alpha'(0)$.
Since its based homotopy class agrees with that of
$\ell_{\alpha'}$, there exists a homotopy
\[
H:I\times I\longrightarrow A_K
\]
from $\ell_\alpha$ to $\ell_{\alpha'}$ such that
\[
H(0,s)=H(1,s)=K(0,s)
\]
for every $s\in I$.
Thus $(H,K)$ gives a homotopy in
$L\Pi(E(K),A_K,\ast;P_1)$ from
\[
(\ell_\alpha,\alpha)
\]
to
\[
(\ell_{\alpha'},\alpha').
\]
Hence $\sigma_1$ is well defined.

Clearly,
\[
\varpi_1\circ\sigma_1
=
\operatorname{id}_{\pi(E(K),A_K,\ast)}.
\]

Conversely, let
\[
[\ell,\alpha]
\in
L\pi(E(K),A_K,\ast;P_1).
\]
Since $\ell$ is freely homotopic to a positive meridian and
$A_K$ has abelian fundamental group, $\ell$ is based homotopic
to $\ell_\alpha$ at the point $\alpha(0)$.
Therefore
\[
[\ell,\alpha]
=
[\ell_\alpha,\alpha],
\]
and hence
\[
\sigma_1\circ\varpi_1
=
\operatorname{id}_{L\pi(E(K),A_K,\ast;P_1)}.
\]

Thus $\varpi_1$ is a bijection with inverse $\sigma_1$.
\end{proof}

\begin{proposition}\label{prop:knot-quandle}
The knot loop quandle
\[
LQ(K,\ast;P_1)
\]
is naturally isomorphic to the knot quandle
\[
Q(K,\ast).
\]
\end{proposition}

\begin{proof}
Let $[\nu]\in Q(K,\ast)$, where $\nu$ is a positive meridional
noose for $K$.
Let
\[
d=d_\nu:D^2\to S^3
\quad\text{and}\quad
\alpha=\alpha_\nu:I\to S^3
\]
be the maps associated with $\nu$ as in the previous section.

Up to homotopy through positive meridional nooses, we may assume
that $d$ is a meridian disk of $K$ with respect to the tubular
neighborhood $N(K)$ and that $\alpha$ is a path in $E(K)$.
Using the product structure of $N(K)$, this standardization can be
carried out continuously along a homotopy through positive
meridional nooses.
Then
\[
\partial d:I\to A_K
\]
is a positive meridian loop of $K$, and hence belongs to $P_1$.
Thus
\[
(\partial d,\alpha)\in L\Pi(E(K),A_K,\ast;P_1).
\]
Define
\[
\eta:Q(K,\ast)\longrightarrow LQ(K,\ast;P_1)
\]
by
\[
\eta([\nu])=[\partial d,\alpha].
\]

If $\nu$ and $\nu'$ are homotopic through positive meridional
nooses, then the corresponding pairs
$(\partial d,\alpha)$ and $(\partial d',\alpha')$ are homotopic
in $L\Pi(E(K),A_K,\ast;P_1)$.
Hence $\eta$ is well defined.

We next verify that $\eta$ preserves the quandle operations.
Let $[\nu_i]\in Q(K,\ast)$ $(i=1,2)$, and let
\[
d_i=d_{\nu_i},
\qquad
\alpha_i=\alpha_{\nu_i}.
\]
By the definition of the quandle operation on $Q(K,\ast)$,
the noose $\nu_1\rhd\nu_2$ is determined by $d_1$ and
\[
\alpha_1\alpha_2^{-1}(\partial d_2)\alpha_2.
\]
Therefore
\[
\begin{aligned}
\eta([\nu_1]\rhd[\nu_2])
&=
\eta([\nu_1\rhd\nu_2])\\
&=
[\partial d_1,
\alpha_1\alpha_2^{-1}(\partial d_2)\alpha_2]\\
&=
[\partial d_1,\alpha_1]
\rhd
[\partial d_2,\alpha_2]\\
&=
\eta([\nu_1])\rhd\eta([\nu_2]).
\end{aligned}
\]
Similarly,
\[
\eta([\nu_1]\,\overline{\rhd}\,[\nu_2])
=
\eta([\nu_1])\,\overline{\rhd}\,\eta([\nu_2]).
\]
Thus $\eta$ is a quandle homomorphism.

Conversely, let
\[
[\ell,\alpha]\in LQ(K,\ast;P_1).
\]
Since $\ell\in P_1$, the loop $\ell$ is freely homotopic in $A_K$
to a positive meridian loop.
Using this free homotopy, we may replace the representative
$(\ell,\alpha)$ by a homotopic pair for which $\ell$ itself is a
positive meridian loop.
Let
\[
d:D^2\longrightarrow N(K)
\]
be a positive meridian disk bounded by $\ell$.
The maps $d$ and $\alpha$ determine a positive meridional noose
$\nu$ for $K$.
Define
\[
\xi:
LQ(K,\ast;P_1)
\longrightarrow
Q(K,\ast)
\]
by
\[
\xi([\ell,\alpha])=[\nu].
\]

We show that $\xi$ is well defined.
Suppose that
\[
[\ell,\alpha]=[\ell',\alpha']
\]
in $LQ(K,\ast;P_1)$, where $\ell$ and $\ell'$ are positive
meridian loops.
Let
\[
d,d':D^2\longrightarrow N(K)
\]
be positive meridian disks bounded by $\ell$ and $\ell'$,
respectively.
Let
\[
H:I\times I\longrightarrow A_K,
\qquad
K:I\times I\longrightarrow E(K)
\]
be a homotopy from $(\ell,\alpha)$ to
$(\ell',\alpha')$ in
$L\Pi(E(K),A_K,\ast;P_1)$.

After modifying the homotopy $H$ within $A_K$, if necessary,
we may assume that $H(\,\cdot\,,s)$ is a positive meridian loop
for every $s\in I$.
Using the product structure of the tubular neighborhood $N(K)$,
these loops bound a continuous family of positive meridian disks
connecting $d$ to $d'$, up to homotopy through positive meridian
disks.
Together with the paths $K(\,\cdot\,,s)$, this family determines
a homotopy through positive meridional nooses from the noose
determined by $d$ and $\alpha$ to the noose determined by $d'$
and $\alpha'$.

Moreover, any two choices of a positive meridian disk bounded by
the same positive meridian loop are homotopic through positive
meridian disks relative to the boundary.
Hence the element of $Q(K,\ast)$ obtained from $[\ell,\alpha]$
does not depend on the choices of the representative or the
positive meridian disk.
Therefore $\xi$ is well defined.

By construction,
\[
\xi\circ\eta
=
\operatorname{id}_{Q(K,\ast)}
\]
and
\[
\eta\circ\xi
=
\operatorname{id}_{LQ(K,\ast;P_1)}.
\]
Therefore $\eta$ is a quandle isomorphism.
\end{proof}

We next recall the power operation on a quandle.
Let $Q=(Q,\rhd,\overline{\rhd})$ be a quandle.
For each $a\in Q$, define
\[
S_a:Q\to Q,\qquad S_a(x)=x\rhd a.
\]
Then $S_a$ is a bijection and
\[
S_a^{-1}(x)=x\,\overline{\rhd}\,a.
\]

For an integer $n$, define binary operations
$\rhd^{[n]}$ and $\overline{\rhd^{[n]}}$ on $Q$ by
\[
x\rhd^{[n]}y=S_y^n(x)
\]
and
\[
x\,\overline{\rhd^{[n]}}\,y=S_y^{-n}(x).
\]
Then
\[
Q^{[n]}
=
\bigl(
Q,\rhd^{[n]},\overline{\rhd^{[n]}}
\bigr)
\]
is a quandle.

This construction is usually referred to as the $n$-th power
operation on a quandle; see Szymik~\cite{Szymik}.
See also Ulrickson~\cite{Ulrickson} for a characterization of
power operations in terms of endofunctors of the category of quandles.
We call $Q^{[n]}$ the \textit{$n$-th power quandle of $Q$}.

For an integer $n$, let $P_n$ be the set of loops in $A_K$ that
are freely homotopic to the $n$-th power of a positive meridian
loop of $K$.
Then $P_n$ is a loop condition on $A_K$. 
Note that, when $n=0$, $P_0$ is the loop condition on $A_K$ consisting of all loops
that are freely null-homotopic in $A_K$ as before.

\begin{proposition}\label{prop:power-knot-quandle}
For every integer $n$, the knot loop quandle
\[
LQ(K,\ast;P_n)
\]
is naturally isomorphic to the $n$-th power quandle
\[
Q(K,\ast)^{[n]}
\]
of the knot quandle.
\end{proposition}

\begin{proof}
We first consider the case $n=0$.
By Propositions~\ref{prop:knot-loop-forgetting-bijection} and
\ref{prop:knot-loop-forgetting-bijection-P1}, the map
\[
\varpi_1^{-1}\circ\varpi_0:
LQ(K,\ast;P_0)
\longrightarrow
LQ(K,\ast;P_1)
\]
is a natural bijection of the underlying sets.

The knot loop quandle $LQ(K,\ast;P_0)$ is a trivial quandle.
Indeed, by Proposition~\ref{prop:loop-forgetting-bijection},
every element has a representative of the form
\[
[c_x,\alpha],
\]
where $c_x$ is a constant loop.
Consequently, for any
$a,b\in LQ(K,\ast;P_0)$,
\[
a\rhd b=a
\quad\text{and}\quad
a\,\overline{\rhd}\,b=a.
\]

On the other hand, by definition,
$Q(K,\ast)^{[0]}$ is the trivial quandle on the underlying set
of $Q(K,\ast)$.
By Proposition~\ref{prop:knot-quandle}, the underlying set
of $Q(K,\ast)$ is naturally identified with that of
$LQ(K,\ast;P_1)$.
Hence
\[
LQ(K,\ast;P_0)
\cong
Q(K,\ast)^{[0]}.
\]

We now consider the case $n\neq 0$.
By Proposition~\ref{prop:knot-quandle}, we identify
$Q(K,\ast)$ with $LQ(K,\ast;P_1)$ as quandles.

Let
\[
[\ell,\alpha]\in LQ(K,\ast;P_1).
\]
Then $\ell$ is freely homotopic to a positive meridian loop of
$K$, and hence $\ell^n$ is freely homotopic to the $n$-th power
of a positive meridian loop.
Here and throughout, negative powers of loops are understood using
the inverse loop.
Thus $\ell^n\in P_n$.

Define
\[
\eta_n:
LQ(K,\ast;P_1)^{[n]}
\longrightarrow
LQ(K,\ast;P_n)
\]
by
\[
\eta_n([\ell,\alpha])
=
[\ell^n,\alpha].
\]

We first show that $\eta_n$ is well defined.
Suppose that $(\ell,\alpha)$ and $(\ell',\alpha')$ are homotopic
in $L\Pi(E(K),A_K,\ast;P_1)$.
Then there exist homotopies
\[
H:I\times I\to A_K
\quad\text{and}\quad
K:I\times I\to E(K)
\]
such that, for every $s\in I$,
\[
H(0,s)=H(1,s)=K(0,s),
\qquad
K(1,s)=\ast,
\]
and
\[
(H(\,\cdot\,,0),K(\,\cdot\,,0))
=
(\ell,\alpha),
\qquad
(H(\,\cdot\,,1),K(\,\cdot\,,1))
=
(\ell',\alpha').
\]
Taking the $n$-th power of the loop $H(\,\cdot\,,s)$ for each
$s\in I$ gives a homotopy from $\ell^n$ to $(\ell')^n$,
with the base point moving along $K(0,s)$.
Hence
\[
(\ell^n,\alpha)
\quad\text{and}\quad
((\ell')^n,\alpha')
\]
are homotopic in $L\Pi(E(K),A_K,\ast;P_n)$.
Therefore $\eta_n$ is well defined.

We next compare the quandle operations.
Let
\[
x=[\ell_1,\alpha_1],
\qquad
y=[\ell_2,\alpha_2]
\]
be elements of $LQ(K,\ast;P_1)$.
For every integer $n$,
\[
S_y^n(x)
=
[\ell_1,
\alpha_1\alpha_2^{-1}\ell_2^n\alpha_2].
\]
Consequently,
\[
x\rhd^{[n]}y
=
[\ell_1,
\alpha_1\alpha_2^{-1}\ell_2^n\alpha_2].
\]
Therefore
\[
\begin{aligned}
\eta_n(x\rhd^{[n]}y)
&=
[\ell_1^n,
\alpha_1\alpha_2^{-1}\ell_2^n\alpha_2]
\\
&=
[\ell_1^n,\alpha_1]
\rhd
[\ell_2^n,\alpha_2]
\\
&=
\eta_n(x)\rhd\eta_n(y).
\end{aligned}
\]
Similarly,
\[
\eta_n
\bigl(
x\,\overline{\rhd^{[n]}}\,y
\bigr)
=
\eta_n(x)\,\overline{\rhd}\,\eta_n(y).
\]
Thus $\eta_n$ is a quandle homomorphism.

It remains to prove that $\eta_n$ is bijective.
Let
\[
[\lambda,\alpha]\in LQ(K,\ast;P_n).
\]
After replacing the representative by a homotopic one, we may
assume that
\[
\lambda=\ell^n,
\]
where $\ell$ is a positive meridian loop based at $\alpha(0)$.
Define
\[
\xi_n([\ell^n,\alpha])
=
[\ell,\alpha].
\]

We verify that $\xi_n$ is well defined.
Suppose that
\[
[\ell^n,\alpha]
=
[(\ell')^n,\alpha']
\]
in $LQ(K,\ast;P_n)$.
Let $(H,K)$ be a homotopy realizing this equivalence, and put
\[
\beta(s)=K(0,s).
\]
After changing the base point along $\beta$, the homotopy $H$
implies
\[
[\beta^{-1}\ell^n\beta]
=
[(\ell')^n]
\]
in $\pi_1(A_K,\alpha'(0))$.
Since
\[
[\beta^{-1}\ell^n\beta]
=
[\beta^{-1}\ell\beta]^n
\]
and
\[
\pi_1(A_K,\alpha'(0))
\cong
\mathbb{Z}^2,
\]
the map
\[
g\longmapsto g^n
\]
is injective for $n\neq0$.
Hence
\[
[\beta^{-1}\ell\beta]
=
[\ell']
\]
in $\pi_1(A_K,\alpha'(0))$.
It follows that
\[
[\ell,\alpha]
=
[\ell',\alpha']
\]
in $LQ(K,\ast;P_1)$.
Thus $\xi_n$ is well defined.

By construction,
\[
\xi_n\circ\eta_n
=
\operatorname{id}_{LQ(K,\ast;P_1)}
\]
and
\[
\eta_n\circ\xi_n
=
\operatorname{id}_{LQ(K,\ast;P_n)}.
\]
Therefore $\eta_n$ is an isomorphism.

Combining this with Proposition~\ref{prop:knot-quandle}, we obtain
\[
LQ(K,\ast;P_n)
\cong
LQ(K,\ast;P_1)^{[n]}
\cong
Q(K,\ast)^{[n]}.
\]
\end{proof}

We now describe more general loop conditions on $A_K$.
Suppose that $K$ is oriented.
Let $\mu$ be a positive meridian and let $\lambda$ be the
preferred longitude of $K$.
Using the ordered basis $(\mu,\lambda)$, the set of free homotopy
classes of loops in
\[
A_K\cong S^1\times S^1
\]
is naturally identified with $\mathbb{Z}\times\mathbb{Z}$.
More precisely, the free homotopy class of a loop $\ell$ in $A_K$
corresponds to $(p,q)\in\mathbb{Z}\times\mathbb{Z}$ if $\ell$ is
freely homotopic to
\[
\mu^p\lambda^q.
\]

\begin{definition}
Let $S$ be a subset of $\mathbb{Z}\times\mathbb{Z}$.
Let $P_S$ be the set of loops in $A_K$ whose free homotopy
classes correspond to elements of $S$.
Then $P_S$ is a loop condition on $A_K$, which we call the
\textit{loop condition corresponding to $S$}.

We also write
\[
LQ(K,\ast;S)
=
LQ(K,\ast;P_S)
\]
and call it the \textit{knot loop quandle of $K$ with respect to
$S$}.
\end{definition}

For example,
\[
P_1=P_{\{(1,0)\}},
\qquad
P_n=P_{\{(n,0)\}}.
\]
Thus, for every subset
$S\subset\mathbb{Z}\times\mathbb{Z}$, we obtain a knot loop
quandle $LQ(K,\ast;S)$.

\begin{example}
Let $K$ be an oriented knot, and let
\[
S=\{(1,0),(-1,0)\}
\subset\mathbb{Z}\times\mathbb{Z}.
\]
Then the knot loop quandle
\[
LQ(K,\ast;S)
\]
is naturally isomorphic to the full knot quandle
\[
\widetilde{Q}(K,\ast).
\]
This follows by the same argument as in
Proposition~\ref{prop:knot-quandle}, allowing both positive and
negative meridional loops.
\end{example}

\section{Concluding remarks}
\label{sec:concluding-remarks}

The essential idea behind the loop fundamental quandle may be
understood by comparing it with Joyce's noose construction.
Recall that a noose in Joyce's construction is defined on
\[
N=D^2\cup T.
\]
At the level of the noose picture, our construction may be viewed
as retaining the boundary-loop data of the disk part rather than
requiring the disk itself.

More precisely, in the knot setting, after passing from the pair
$(S^3,K)$ to the exterior pair
\[
(E(K),A_K),
\]
the boundary of a meridian disk becomes a loop in $A_K$.
The loop fundamental quandle uses such a loop, together with a path
from its base point to the fixed base point $\ast$, as its basic
data.
Unlike Joyce's noose construction, however, the loop fundamental
quandle does not require the loop to extend to a disk satisfying
the noose condition.

Some familiar quandles arise naturally in this framework.
In Propositions~\ref{prop:knot-quandle} and
\ref{prop:power-knot-quandle}, we have seen that the knot quandle
$Q(K,\ast)$ is naturally isomorphic to
\[
LQ(K,\ast;P_1),
\]
and that the $n$-th power quandle $Q(K,\ast)^{[n]}$ is naturally
isomorphic to
\[
LQ(K,\ast;P_n).
\]
Thus the knot quandle and its power quandles arise as particular
cases of the knot loop quandle.

On the other hand, the present construction also allows loop
conditions that are not represented by Joyce nooses.
Suppose that a loop condition $P$ contains a loop
\[
\ell:I\longrightarrow A_K
\]
for which there is no map
\[
d:D^2\longrightarrow S^3
\]
such that
\[
d\circ\delta=\ell
\]
and
\[
d^{-1}(K)=\{O\}.
\]
Then an element of $LQ(K,\ast;P)$ represented by a pair
$[\ell,\alpha]$ has no corresponding representative given by a
Joyce noose with boundary loop $\ell$.
In this sense, the knot loop quandle provides a more flexible
framework than the one obtained directly from Joyce's
disk-based noose construction.

Although in this paper we have defined the knot loop quandle for
a knot in $S^3$, the construction extends naturally to
codimension-two submanifolds of arbitrary manifolds.
Let $K$ be a codimension-two submanifold of a manifold $M$, let
$N(K)$ be a tubular neighborhood of $K$, and put
\[
E(K)=M\setminus\operatorname{Int}N(K),
\qquad
A_K=\partial E(K).
\]
Fix a base point $\ast\in E(K)$.
For any loop condition $P$ on $A_K$, one may consider the loop
fundamental quandle
\[
LQ(E(K),A_K,\ast;P)
\]
as an analogue of the knot loop quandle.

In particular, when $K$ is a surface-knot in $\mathbb{R}^4$,
there are various natural choices of loop conditions on $A_K$.
Investigating the quandles arising from these different loop
conditions may provide further applications of the present
construction to higher-dimensional knot theory.

\section*{Tool Use}

The author used ChatGPT (OpenAI) to assist with English language
editing, organization of the manuscript, LaTeX formatting, and
the presentation of some mathematical arguments.
All mathematical statements and proofs were independently checked
by the author, who takes full responsibility for the contents of
the paper.


\begin{thebibliography}{9}

\bibitem{FennRourke92}
R. Fenn and C. Rourke,
Racks and links in codimension two,
\textit{J. Knot Theory Ramifications} {\bf 1} (1992), 343--406.

\bibitem{Joyce82}
D. Joyce,
A classifying invariant of knots, the knot quandle,
\textit{J. Pure Appl. Algebra} {\bf 23} (1982), 37--65.



\bibitem{Kamada07}
S.~Kamada,
Quandles with good involutions, their homologies and knot invariants,
in \textit{Intelligence of Low Dimensional Topology 2006},
eds. J.~S. Carter et al.,
pp.~101--108,
World Scientific Publishing Co., 2007.

\bibitem{KamadaOshiro10}
S.~Kamada and K.~Oshiro,
Homology groups of symmetric quandles and cocycle invariants of links
and surface-links,
\textit{Trans. Amer. Math. Soc.} {\bf 362} (2010), no.~10, 5501--5527.
arXiv:0902.4277.

\bibitem{Kamada14}
S.~Kamada,
Quandles and symmetric quandles for higher dimensional knots,
in \textit{Knots in Poland III -- Part III},
\textit{Banach Center Publ.} {\bf 103} (2014), 145--158.
arXiv:1403.0689.


\bibitem{Matveev82}
S.~V. Matveev,
Distributive groupoids in knot theory.
(Russian)
\textit{Mat. Sb. (N.S.)} {\bf 119 (161)} (1982), no.~1, 78--88;
\textit{Math. USSR-Sb.} {\bf 47} (1984), 73--83 (in English).


\bibitem{Szymik}
M.~Szymik,
Permutations, power operations, and the center of the category of racks,
\textit{Comm. Algebra} {\bf 46} (2018), no.~1, 230--240.

\bibitem{Ulrickson}
P.~Ulrickson,
Endofunctors of quandles and racks,
\textit{Comm. Algebra} {\bf 46} (2018), no.~7, 2964--2967.



\end{thebibliography}
\end{document}